\documentclass[a4paper, reqno, 11pt]{article}
\usepackage[utf8]{inputenc} 

\usepackage{scalerel}
\usepackage{bm}
\usepackage{amsthm}
\usepackage{amsmath}
\usepackage{amssymb}
\usepackage[american]{babel}
\usepackage{exscale}
\usepackage{verbatim}
\usepackage[mathscr]{euscript}
\usepackage{url}
\usepackage[all,ps,tips,tpic]{xy}
\usepackage{graphicx}
\usepackage{mathdots}
\usepackage{color}
\usepackage{stmaryrd}
\usepackage{framed}
\usepackage{enumitem}
\usepackage{mathrsfs}
\usepackage{dsfont}
\usepackage{tikz-cd}
\usepackage{hyperref}

\SetLabelAlign{LeftAlignWithIndent}{\hspace*{2.0ex}\makebox[1.5em][l]{#1}}

\newtheorem{lem}{Lemma}[section]

\newtheorem{thm}[lem]{Theorem}
\newtheorem{prop}[lem]{Proposition}

\theoremstyle{remark}

\theoremstyle{remark}

\DeclareMathOperator{\im}{im}

\DeclareMathOperator{\tr}{tr}

\newcommand{\Irr}{\mathrm{Irr}}

\newcommand{\GL}{\mathrm{GL}}

\newcommand{\reg}{\mathrm{reg}}

\newcommand{\mf}[1]{\mathfrak{#1}}
\newcommand{\ms}[1]{\mathscr{#1}}
\newcommand{\mc}[1]{\mathcal{#1}}
\newcommand{\mb}[1]{\mathbf{#1}}
\newcommand{\Leta}{{{}^{L}\eta}}

\newcommand{\LG}{{^{L}G}}

\newcommand{\ov}[1]{\overline{#1}}

\newcommand{\bb}[1]{\mathbb{#1}}
\newcommand{\Inn}{\mathrm{Inn}}

\newcommand{\ad}{\mathrm{ad}}

\newcommand{\Q}{\mathbb{Q}}

\newcommand{\C}{\bb{C}}

\newcommand{\inv}{\mathrm{inv}}

\newcommand{\Aut}{\mathrm{Aut}}

\newcommand{\iso}{\mathrm{iso}}

\newcommand{\der}{\mathrm{der}}

\newcommand{\Ad}{\mathrm{Ad}}

\newcommand{\LH}{{^{L}H}}

\newcommand{\SC}{\mathrm{sc}}

\newcommand{\el}{\mathrm{ell}}

\newcommand{\bas}{\mathrm{bas}}
\newcommand{\Avg}{\mathrm{Avg}}

\newdir{ >}{{}*!/-5pt/@{>}}
\SelectTips{cm}{10}

\title{An Approach to the Characterization of the Local Langlands Correspondence}
\author{Alexander Bertoloni Meli \\ Alex Youcis}
\begin{document}
\maketitle

\section{Introduction}

The local Langlands conjecture for a reductive group $G$ over a $p$-adic field $F$ has held a central role in the study of number theory since its initial development by R. Langlands. While the precise formulation of these conjectures for the group $G=\GL_{n,F}$ is classical (e.g. see \cite{HarrisTaylor} and \cite{Henniart}), such statements for general $G$, and especially for $G$ which are not quasi-split, have only gradually been made precise over recent years (e.g. see \cite{KalethaLLNQSG} and the references therein). Such statements often list desiderata that the local Langlands conjecture for $G$ is expected to satisfy but generally make no claim that these properties uniquely characterize the correspondence.

In the case of $G=\GL_{n,F}$ such characterizations classically employ the theory of $L$-functions and $\epsilon$-factors. For other classical groups $G$, Arthur realized that one can often use the theory of standard and twisted endoscopy to relate the local Langlands conjecture for $G$ and the local Langlands conjecture for $\GL_{n,F}$, thus reducing the characterization problem for $G$ to the case of $\GL_{n,F}$. This was carried out for quasi-split symplectic and orthogonal groups in \cite{Arthurbook} and for quasi-split unitary groups in \cite{Mok}. The non quasi-split unitary case was tackled in \cite{KMSW}.

In \cite{ScholzeMasters}, Scholze gave an alternate characterization of the local Langlands conjecture for $\GL_{n,F}$. This characterization involves an explicit equation (called the \emph{Scholze--Shin equation(s)} in the article below) which relates the local Langlands conjecture to certain geometrically defined functions $f_{\tau,h}$ which are of geometric provenance. This characterization, unlike that appealing to the theory of $L$-functions and $\epsilon$-factors, has the property that it is amenable to study for a general group $G$. Namely, the functions $f_{\tau,h}$ as in \cite{ScholzeMasters} can be defined for a much wider class of groups than just $\GL_{n,F}$ (e.g. see \cite{ScholzeDeformation} and \cite{Youcis}) and thus one can ask whether analogues of Scholze's characterization of the local Langlands conjecture for $\GL_{n,F}$ exist for other groups $G$.

The goal of this article can then be stated as giving an affirmative answer to this question for supercuspidal $L$-parameters (conjecturally the class of parameters whose packets consist entirely of supercuspidal representations). We show that a conjectural Langlands correspondence satisfying a certain list of desiderata including the Scholze-Shin equations is uniquely characterized by these conditions. Our method, as currently stated, cannot hope to handle all groups $G$ but only groups satisfying a certain `niceness' condition. For example, both unitary and odd special orthogonal groups are `nice' whereas symplectic and even special orthogonal groups are not in general `nice'. 

\subsection*{Acknowledgements}

The authors are greatly indebted to Kaoru Hiraga for describing the argument for Proposition \ref{hiragalem}. This proposition allowed the authors to remove `atomic stability for singeltons' as an assumption in Theorem \ref{main}. 

The authors would like to kindly thank Tasho Kaletha, Sug Woo Shin, and Masao Oi for very helpful conversations concerning this paper. Finally, the authors would like to thank Mr.\@ and Mrs.\@ Bertoloni Meli for providing an excellent environment to write a paper replete with apples and rapini.

During the completion of this work, the first author was partially funded by NSF RTG grant 1646385. The second author was partially supported by the funding from the European Research Council (ERC)
under the European Union’s Horizon 2020 research and innovation program (grant agreement No. 802787).

\section{Notation}
Let $F$ be a $p$-adic local field. Fix an algebraic closure $\ov{F}$ and let $F^\mathrm{un}$ be the maximal unramified extension of $F$ in $\ov{F}$. Let $L$ be the completion of $F^\mathrm{un}$ and fix an algebraic closure $\ov{L}$.

Let $G$ be a connected reductive group over $F$. We denote by $G(F)^\reg$ the regular semisimple elements in $G(F)$ and by $G(F)^\el$ the subset of elliptic regular semisimple elements. We denote by $D$, or $D_G$, the discriminant map on $G(F)$. If $\gamma,\gamma'\in G(F)$ are stably conjugate we denote this by $\gamma\sim_\mathrm{st}\gamma'$.

Let $\widehat{G}$ be the connected Langlands dual group of $G$ and let $\LG$ be the Weil group version of the $L$-group of $G$ as defined in \cite[\S 1]{KottwitzCuspidal}. We denote the set of irreducible smooth representations of $G(F)$ by $\Irr(G(F))$ and by $\Irr^\SC(G(F))$ the subset of supercuspidal representations. For a finite group $C$ the notation $\Irr(C)$ means all irreducible $\C$-valued representations of $C$.

A \emph{supercuspidal Langlands parameter} is an $L$-parameter (see \cite[\S8.2]{BorelCorvallis}) $\psi: W_F \to \LG$ such that the image of $\psi$ is not contained in a proper Levi subgroup of $\LG$. We say that supercuspidal parameters $\psi$ and $\psi'$ are equivalent if they are conjugate in $\widehat{G}$ and denote this by $\psi \sim \psi'$. Let $C_{\psi}$ be the centralizer of $\psi(W_F)$ in $\widehat{G}$. Then by \cite[\S 10.3.1]{KottwitzCuspidal}, $\psi$ is supercuspidal if and only if the identity component $C^{\circ}_{\psi}$ of $C_{\psi}$ is contained in $Z(\widehat{G})^{\Gamma_F}$. We define the group $\ov{C_{\psi}} := C_{\psi}/ Z(\widehat{G})^{\Gamma_F}$ which is finite by our assumptions on $\psi$. For the sake of comparison, in \cite[Conj. F]{KalethaLLNQSG}, Kaletha defines $S^{\natural}_{\psi}:= C_{\psi} / (C_{\psi} \cap [\widehat{G}]_{\der})^{\circ}$. For $\psi$ a supercuspidal parameter, we have
\begin{equation}
S^{\natural}_{\psi} = C_{\psi}.  
\end{equation}
Indeed,
\begin{equation}
    (C_{\psi} \cap [\widehat{G}]_{\der})^{\circ}=(C^{\circ}_{\psi} \cap [\widehat{G}]_{\der})^{\circ} \subset (Z(\widehat{G})^{\Gamma_F} \cap [\widehat{G}]_{\der})^{\circ}=\{1\},
\end{equation}
from where the equality follows. 

Define $Z^1(W_F, G(\ov{L}))$ to be the set of continuous cocycles of $W_F$ valued in $G(\ov{L})$ and let $\mb{B}(G) := H^1(W_F, G(\ov{L}))$ be the corresponding cohomology group. Let $\kappa: \mb{B}(G) \to X^*(Z(\widehat{G})^{\Gamma_F})$ be the Kottwitz map as in \cite{Kotiso}.

An \emph{elliptic endoscopic datum} of  $G$ (cf. \cite[7.3-7.4]{KottwitzCuspidal}) is a triple $(H, s, \eta)$ of a quasisplit reductive group $H$, an element $s \in Z(\widehat{H})^{\Gamma_F}$, and a homomorphism $\eta : \widehat{H} \to \widehat{G}$. We require that $\eta$ gives an isomorphism
\begin{equation}
    \eta: \widehat{H} \to Z_{\widehat{G}}(\eta(s))^{\circ},
\end{equation}
that the $\widehat{G}$-conjugacy class of $\eta$ is stable under the action of $\Gamma_F$, and that $(Z(\widehat{H})^{\Gamma_F})^{\circ} \subset Z(\widehat{G})$.

An \emph{extended elliptic endoscopic datum} of $G$ is a triple $(H,s, \Leta)$ such that $\Leta: \LH \to \LG$ and $(H, s, \Leta|_{\widehat{H}})$ gives an elliptic endoscopic datum of $G$.

An \emph{extended elliptic hyperendoscopic datum} is a sequence of tuples of data $(H_1, s_1, \Leta_1),\ldots,(H_k, s_k, \Leta_k)$ such that $(H_1, s_1, \Leta_1)$ is an extended elliptic endoscopic datum of $G$, and for $i>1$, the tuple $(H_i, s_i, \Leta_i)$ is an extended elliptic endoscopic datum of $H_{i-1}$. An \emph{elliptic hyperendoscopic group} of $G$ is a quasisplit connected reductive group $H_k$ appearing in an extended elliptic hyperendoscopic datum for $G$ as above.

\section{Statement of main result}

Throughout the rest of the paper we assume that our groups $G$ satisfy the following assumption:
\begin{itemize}[align=parleft,labelsep=.5cm]
\item[\textbf{(Ext)}] For each elliptic hyperendoscopic group $H$ of $G$ and each elliptic endoscopic datum  $(H',s,\eta')$ of $H$, one can extend $(H', s, \eta')$ to an extended elliptic endoscopic datum $(H', s, \Leta')$ such that $\Leta': \LH' \to \LH$.
\end{itemize}
For a discussion on the severity of this assumption see \S\ref{extcondsec}.

We now state the main result. Let us fix $G^\ast$ to be a quasi-split reductive group over $F$. We define a \emph{supercuspidal local Langlands correspondence} for a group $G^\ast$ to be an assignment 
\begin{equation}
    \Pi_H:\left\{\begin{matrix}\text{Equivalence classes of}\\ \text{Supercuspidal }L\text{-parameters} \\ \text{for } H\end{matrix}\right\} \to \left\{\begin{matrix}\text{Subsets of}\\ \Irr^\SC(H(F))\end{matrix}\right\},
\end{equation}
for every elliptic hyperendoscopic group $H$ of $G^\ast$ satisfying the following properties:
\begin{itemize}[align=parleft,labelsep=.5cm]
    \item[\textbf{(Dis)}] If $\Pi_H(\psi) \cap \Pi_H(\psi') \neq \emptyset$ then $\psi \sim \psi'$.
    \item[\textbf{(Bij)}] For each Whittaker datum $\mf{w}_H$ of $H$, a bijection
    \begin{equation}
        \iota_{\mf{w}_H}:\Pi_H(\psi)\to \Irr(\ov{C_\psi}).
    \end{equation}
    This bijection $\iota_{\mf{w}_H}$ gives rise to a pairing
    \begin{equation}
        \langle -,-\rangle_{\mf{w}_H}:\Pi_H(\psi)\times \ov{C_\psi}\to \C,
    \end{equation}
    defined as follows:
    \begin{equation}
        \langle \pi,s\rangle_{\mf{w}_H}:=\tr(s \mid \iota_{\mf{w}_H}(\pi)).
    \end{equation}
    \item[\textbf{(St)}] For all supercuspidal $L$-parameters $\psi$ of $H$, the distribution
    \begin{equation}
        S\Theta_\psi:=\sum_{\pi\in\Pi_H(\psi)}\langle \pi,1\rangle\Theta_\pi,
    \end{equation}
    is stable and does not depend on the choice of $\mf{w}_H$.
    \item[\textbf{(ECI)}] For all extended elliptic endoscopic data $(H',s, \Leta)$ for $H$ and all $h\in \mathscr{H}(H(F))$, suppose $\psi^H$ is a supercuspidal $L$-parameter of $H$ that factors through $\Leta$ by some parameter $\psi^{H'}$. Then such a $\psi^{H'}$ must be supercuspidal and we assume it satisfies the  \emph{endoscopic character identity}:
    \begin{equation}
        S\Theta_{\psi^{H'}}(h^{H'})=\Theta^{s}_{\psi^{H}}(h),
    \end{equation}
    where we define $h^{H'}$ to be a transfer of $h$ to $H'$ (e.g. see \cite[\S1.3]{KalethaLLNQSG}) and we define
    \begin{equation}
        \Theta^{s}_{\psi^{H}}:=\sum_{\pi\in\Pi_H(\psi^{H})}\langle \pi,s\rangle\Theta_{\pi}.
    \end{equation}
    the $s$-twisted character of $\psi^H$.
\end{itemize}
Suppose now that $z_\iso\in Z^1(W_F,G(\ov{L}))$ projecting to an element of $\mb{B}(G)_{\bas}$. Let $G$ be the inner form of $G^\ast$ corresponding to the projection of $z_\iso$ to $Z^1(W_F,\Aut(G)(\ov{F}))$.  We then define a \emph{supercuspidal local Langlands correspondence} for the \emph{extended pure inner twist} $(G,z_\iso)$ (cf. \cite[\S2.5]{KalethaLLNQSG}) to be a supercuspidal local Langlands correspondence for $G^\ast$ as well as a correspondence
\begin{equation}
    \Pi_{(G,z_\iso)}:\left\{\begin{matrix}\text{Supercuspidal }L\text{-parameters}\\ \text{for }G\end{matrix}\right\}\to \left\{\begin{matrix}\text{Subsets of}\\ \Irr^\SC(G(F))\end{matrix}\right\},
\end{equation}
satisfying
\begin{itemize}[align=parleft,labelsep=.5cm]
\item[\textbf{(Bij')}] For each Whittaker datum $\mf{w}_G$ of $G$, a bijection
\begin{equation}
    \iota_{\mf{w}_G}:\Pi_G(\psi)\to \Irr(C_{\psi},\chi_{z_\iso}),
\end{equation}
where $\Irr(C_{\psi}, \chi_{z_{\iso}})$ denotes the set of equivalence classes of irreducible algebraic representations of $C_{\psi}$ with central character on $Z(\widehat{G})^{\Gamma_F}$ equal to $\chi_{z_{\iso}} :=\kappa(\ov{z_{\iso}})$. This gives rise to a pairing
\begin{equation}
    \langle -,-\rangle_{\mf{w}_G}:C_\psi\times\Irr(C_\psi,\chi_{z_\iso})\to\C,
\end{equation}
defined as
\begin{equation}
    \langle \pi,s\rangle_{\mf{w}_G}:=\tr(s \mid \iota_{\mf{w}_G}(\pi)).
\end{equation}
\item[\textbf{(ECI')}] For all supercuspidal parameters $\psi$ of $G$ and all extended elliptic endoscopic data $(H,s,\Leta)$ of $G$ such that $\psi$ factors as $\psi= \Leta \circ \psi^H$, there is an equality 
\begin{equation}
    \Theta^s_{\psi^H}(h^H)=S\Theta_\psi(h),
\end{equation}
where $h\in \ms{H}(G(F))$ and $S\Theta_{\psi}$ is independent of choice of Whittaker datum in \textbf{(Bij')}.
\end{itemize}

For a supercuspidal local Langlands correspondence $\Pi$ for $(G,z_\iso)$ we say that a subset of $\Irr(H(F))$ of the form $\Pi_\psi(H)$ is a \emph{supercuspidal $L$-packet} for $\Pi_H$. We furthermore say that an element $\pi$ of $\Irr(H(F))$ is \emph{$\Pi_H$-accessible} if $\pi$ is in a supercuspidal $L$-packet for $\Pi_H$.

A priori, the above axioms \textbf{(Dis)}, \textbf{(Bij)}-\textbf{(Bij')},\textbf{(St)}, and \textbf{(ECI)}-\textbf{(ECI')} are not enough to uniquely specify a supercuspidal local Langlands correspondence $\Pi$ for $G^*$ even under the specification of the set of $\Pi$-accessible representations. The goal of our main theorem is to explain a sufficient extra condition which does uniquely specify a supercuspidal local Langlands correspondence.

In the statement of this condition we need to assume an extra property of $G$. Namely, we say that $G^*$ is \emph{good} if for every elliptic hyperendoscopic group $H$ of $G^\ast$ we have:
\begin{itemize}[align=parleft,labelsep=.5cm]
\item[\textbf{(Mu)}] There exists a set $S^H$ of dominant cocharacters of $H_{\ov{F}}$ with the following propery. Let $\psi^H_1, \psi^H_2$ be any pair of supercuspidal parameters of $H$ such that for all dominant cocharacters $\mu\in S^H$, we  have an equivalence $r_{-\mu} \circ \psi^{H}_1 \sim r_{- \mu} \circ \psi^H_2$. Then $\psi^H_1 \sim \psi^H_2$.
\end{itemize}
Here $r_{-\mu}$ is the representation of $\LH$ as defined in \cite[(2.1.1)]{KottwitzTOO}.We say that $G$ is \emph{good} if $G^\ast$ is. We call a set $S^H$ as in assumption \textbf{(Mu)} \emph{sufficient}. See \S\ref{goodsec} for a discussion of the severity of this assumption.

To this end, let us define a \emph{Scholze--Shin datum} $\{f_{\tau,h}^{\mu}\}$ for $G$ to consist of the following data for each elliptic hyperendoscopic group $H$ of $G$:
\begin{itemize}
    \item A compact open subgroup $K^H \subset H(F)$,
    \item A sufficient set $S^H$ of dominant cocharacters of $H_{\ov{F}}$,
    \item For each $\mu \in S^H$ of with reflex field $E_{\mu}$, each $\tau \in W_{E_{\mu}}$, and each $h \in \ms{H}(K^H)$, a function $f^{\mu}_{\tau,h} \in \ms{H}(H(F))$.
\end{itemize}

Let us say that a supercuspidal local Langlands correspondence for $G$ satisfies the \emph{Scholze--Shin equations} relative to the Scholze--Shin datum $\{f_{\tau,h}^{\mu}\}$ if the following holds:
\begin{itemize}[align=parleft,labelsep=.5cm]
    \item[\textbf{(SS)\text{ }}] For all elliptic hyperendoscopic groups $H$, all $h\in\ms{H}(K^H)$, all ${\mu}\in S^H$, and all parameters $\psi^H$ of $H$ one has that
    \begin{equation}
        S\Theta_{\psi^H}(f_{\tau,h}^{\mu})=\tr(\tau\mid (r_{-{\mu}}\circ \psi^H)(\chi_\mu))S\Theta_{\psi^H}(h),
    \end{equation}
    where $\chi_\mu:=|\cdot|^{-\langle \rho,\mu\rangle}$ and $\rho$ is the half-sum of the positive roots of $H$ (for a representation $V$ and character $\chi$ we denote by $V(\chi)$ the character twist of $V$ by $\chi$).
\end{itemize}

We then have the following result:

\begin{thm}\label{main} Let $G$ be a good group and suppose $\Pi^i$ for $i=1,2$ are supercuspidal local Langlands correspondences for $(G,z_\iso)$ such that
\begin{enumerate}
    \item For every elliptic hyperendoscopic group $H$ of $G$ the set of $\Pi^1_H$-accessible representations is contained in the set of $\Pi^2_H$-accessible representations. 
    \item There exists a Scholze--Shin datum $\{f^{\mu}_{\tau,h}\}$ such that $\Pi^i$ satisfies \textbf{\emph{(SS)}} relative to $\{f^{\mu}_{\tau,h}\}$ for $i=1,2$.
\end{enumerate}
Then $\Pi^1=\Pi^2$ and for every $(H, z)$, either equal to $(H,1)$ where $H$ is an elliptic hyperendoscopic group of $G$ or equal to $(G,z_\iso)$, and choice of Whittaker datum $\mf{w}_H$,  the bijections $\iota_{\mf{w}_H}^i$ for $i=1,2$ agree.
\end{thm}

\section{Atomic stability of $L$-packets}

Before we begin the proof of Theorem \ref{main} in earnest, we first discuss the following extra assumption one might make on a supercuspidal local Langlands correspondence $\Pi$ for the group $G$ which, for this section, we assume is quasi-split. Namely, let us say that $\Pi$ possesses \emph{atomic stability} if the following condition holds:
    \begin{itemize}[align=parleft,labelsep=.5cm]
    \item[\textbf{(AS)}] If $S=\{\pi_1, ...,\pi_k\}$ is a finite subset of $\Pi$-accessible elements of $\Irr^\SC(G(F))$ and $\{a_1, ..., a_k\}$ is a set of complex numbers such that $\displaystyle \Theta:=\sum\limits^k_{i=1} a_i \pi_i$ is a stable distribution, then there is a partition
    \begin{equation}
        S=\Pi_{\psi_1}(G)\sqcup\cdots\sqcup \Pi_{\psi_n}(G)
    \end{equation}
    such that
    \begin{equation}
        \Theta=\sum_{j=1}^n b_j S\Theta_{\psi_j}
    \end{equation}
    (i.e. that $a_i$ is constant on $\Pi_{\psi_i}(G)$).
    \end{itemize}
We then have the following result:

\begin{prop}\label{stabprop} Let $\Pi$ be supercuspidal local Langlands correspondence for a group $G$. Then, $\Pi$ automatically possesses atomic stability. 
\end{prop}

Proposition \ref{stabprop} will follow from the following a priori weaker proposition. To state it we make the following definitions. For supercuspidal $L$-parameters $\psi_1, \ldots, \psi_n$ we denote by $D(\psi_1, \ldots, \psi_n)$ the $\C$-span of the distributions $\Theta_\pi$ for $\pi\in \Pi_G(\psi_1)\cup\cdots\cup\Pi_G(\psi_n)$ and let $S(\psi_1, \ldots, \psi_n)$ be the subspace of stable distributions in $D(\psi_1, \ldots , \psi_n)$. 

\begin{prop}\label{hiragalem} For any finite set of supercuspidal $L$-parameters $\{\psi_1, \ldots, \psi_n\}$ one has that $\{S\Theta_{\psi_1}, \ldots, S\Theta_{\psi_n}\}$ is a basis for $S(\psi_1, \ldots,\psi_n)$.
\end{prop}

Let us note that this proposition actually implies Proposition \ref{stabprop}. Indeed, since each $\pi_i\in S$ is accessible we know that we can enlarge $S$ to be a union $\Pi_{\psi_1}(G)\sqcup\cdots\sqcup \Pi_{\psi_n}(G)$ of $L$-packets. Proposition \ref{stabprop} is then clear since every stable distribution in the span of $S$ is contained in $S(\psi_1,\ldots,\psi_n)$. 

Before we proceed with the proof of Proposition \ref{hiragalem} we establish some further notation and basic observations. For an $\pi$ element of $\Irr^\SC(G(F))$ we denote by $f_\pi$ the locally constant $\C$-valued function on $G(F)^\reg$ given by the Harish-Chandra regularity theorem. We then obtain a linear map
\begin{equation}
    R:D(\Irr^\SC(G(F)))\to C^\infty(G(F)^\el,\C)
\end{equation}
given by linearly extending the association $\Theta_\pi\mapsto f_\pi\mid_{G(F)^\el}$. Here $D(\Irr^\SC(G(F)))$ is the $\C$-span of the distributions on $\mc{H}(G(F))$ of the form $\Theta_\pi$ for $\pi\in\Irr^\SC(G(F))$. We also have averaging maps
\begin{equation}
    \Avg:C^\infty(G(F)^\el,\C)\to C^\infty(G(F)^\el,\C)
\end{equation}
given by 
\begin{equation}
    \Avg(f)(\gamma):=\frac{1}{n_\gamma}\sum_{\gamma'}f(\gamma')
\end{equation}
where $\gamma'$ runs over representatives of the conjugacy classes of $G(F)$ stably equal to the conjugacy class of $\gamma$ and $n_\gamma$ is the number of such classes (which is finite since $F$ is a $p$-adic field).

We then have the following well-known lemma concerning $R$:
\begin{lem}[{\cite[Theorem C]{Kazhdanpaper}}]{\label{Rinj}}
The linear map $R$ is injective.
\end{lem}

In addition, we have the following observation concerning the interaction between $R$ and $\Avg$, which follows from the well-known fact that $\Theta$ is stable implies that $R(\Theta)$ is stable:

\begin{lem}\label{avglem} Let $\Theta\in D(\Irr^\SC(G(F)))$ be stable as a distribution. Then, $\Avg(R(\Theta))=R(\Theta)$.
\end{lem}

We may now proceed to the proof of Proposition \ref{hiragalem}:

\begin{proof}(Proposition \ref{hiragalem}) By assumption \textbf{(Bij)}, the set of virtual characters $S\Theta^s_{\psi_i}$, as $s$ runs through representatives for the conjugacy classes in $\ov{C_\psi}$ and $i$ runs through $\{1,\ldots,n\}$, is a basis of $D(\psi_1,\ldots,\psi_n)$. It suffices to show this in the case when $n=1$ in which case it is clear. Indeed, writing just $\psi$ instead of $\psi_1$, we see that it suffices to note that the matrix $(\langle \pi,s\rangle)$, where $\pi$ runs through the elements of $\Pi_\psi(G)$, is unitary, and thus invertible, by the orthogonality of characters.

We next show that for any supercuspidal $L$-parameter $\psi$ and any non-trivial $s$ in $\ov{C_\psi}$ we have that $\Avg(R(S\Theta^s_\psi))=0$. Indeed, we begin by observing that by \cite[Lemma 6.20]{HiragaHiroshi} we have that 
  \begin{equation}
    \Avg(R(S\Theta^s_\psi))(\gamma) = \frac{1}{n_\gamma}\sum_{\gamma'}\sum_{\gamma_H\in X(\gamma')/\sim_\mathrm{st}}\Delta(\gamma_H,\gamma')\left|\frac{D_H(\gamma_H)}{D_G(\gamma')}\right|S\Theta_{\phi^H}(\gamma_H)
    \end{equation}
where here $\gamma'$ travels over the set of conjugacy classes of $G(F)$ stably equal to the conjugacy class of $\gamma$ and, as in loc. cit., $X(\gamma')$ is the set of conjugacy classes in $H(F)$ that transfer to $\gamma$, and $\Delta(\gamma_H,\gamma')$ is the usual transfer factor, and $D$ denotes the discriminant function. 

Let us note that we can rewrite this sum as 
\begin{equation}
     \frac{1}{n_\gamma}\sum_{\gamma_H\in X(\gamma)/\sim_\mathrm{st}}\left(\sum_{\gamma'}\Delta(\gamma_H,\gamma')\left|\frac{D_H(\gamma_H)}{D_G(\gamma')}\right|\right)S\Theta_{\phi^H}(\gamma_H).
\end{equation}
because $X(\gamma')/\sim_\mathrm{st}$ is independent of the choice of $\gamma'$. 

Note that $D_G(\gamma')=D_G(\gamma)$ for all $\gamma'$ stably conjugate to $\gamma$ (since $D_G(\gamma')$ is defined in terms of the characteristic polynomial of $\Ad(\gamma')$) and thus we can further rewrite this as
\begin{equation}
     \frac{1}{n_\gamma}\sum_{\gamma_H\in X(\gamma)/\sim_\mathrm{st}}\left|\frac{D_H(\gamma_H)}{D_G(\gamma)}\right|\left(\sum_{\gamma'}\Delta(\gamma_H,\gamma')\right)S\Theta_{\phi^H}(\gamma_H)
\end{equation}
and so it suffices to show that this inner sum $\displaystyle \sum_{\gamma'}\Delta(\gamma_H,\gamma')$ is zero. 

For $\gamma' \sim_{st} \gamma$, we have
\begin{equation}
    \Delta(\gamma_H, \gamma')=\langle \inv(\gamma, \gamma'), s \rangle \Delta(\gamma_H, \gamma),
\end{equation}
where $ \inv(\gamma, \gamma') \in \mf{K}(I_\gamma/F)^D$ (as in \cite[\S 2.2]{SugWooIgusa}). Since $\gamma$ is elliptic,  $\gamma' \mapsto \inv(\gamma, \gamma')$ gives a bijection between $F$-conjugacy classes in the stable conjugacy class of $\gamma$ and $\mf{K}(I_{\gamma}/F)^D$. Hence
\begin{equation}
    \sum_{\gamma'}\Delta(\gamma_H,\gamma')= \Delta(\gamma_H, \gamma) \sum\limits_{\chi \in \mf{K}(I_{\gamma})^D} \chi(s).
\end{equation}
In particular, it suffices to show that $s$ gives a nontrivial element of $\mf{K}(I_{\gamma}/F)$. Since $(H,s,\eta)$ is a nontrivial elliptic endoscopic datum and $\gamma$ is elliptic, this follows from \cite[Lemma 2.8]{SugWooIgusa}.

Now, since the set $\{S\Theta_{\psi_1},\ldots,S\Theta_{\psi_n}\}$ is independent (by assumption \textbf{(Dis)}) it suffices to show that this set generates $S(\psi_1,\ldots,\psi_n)$. But, this is now clear since if $\Theta\in S(\psi_1,\ldots,\psi_n)$ then we know by Lemma \ref{avglem} that $\Avg(R(\Theta))=R(\Theta)$. On the other hand, writing
\begin{equation}
    \Theta=\sum_{i=1}^n \sum_s a_{is}S\Theta^s_{\psi_i}
\end{equation}
we see from the above discussion, as well as combining assumption \textbf{(St)} with Lemma \ref{avglem}, that
\begin{equation}
    \Avg(R(\Theta))=\sum_{i=1}^n R(S a_{ie}\Theta_{\psi_i})=R\left(\sum_{i=1}^n a_{ie}S\Theta_{\psi_i}\right)
\end{equation}
(identifying $S\Theta_{\psi_i}$ with $S\Theta^e_{\psi_i}$ where $e$ is the identity conjugacy class in $\ov{C_\psi}$). The claim then follows from Lemma \ref{Rinj}.
\end{proof}

\section{Proof of main result}

Let us begin by explaining that it suffices to assume $G$ is quasi-split. Indeed, note that the assumptions of Theorem \ref{main} are also satisfied for $(G, z_{\iso})$ equal to $(G^\ast,1)$ and so, in particular, if we have proven the theorem in the case of $(G^\ast,1)$ then we know that $\Pi^1_{G^\ast}=\Pi^2_{G^\ast}$. Now, let $\psi$ be any supercuspidal $L$-parameter for $G$. By assumption \textbf{(ECI')} we have that 
\begin{equation}
S\Theta^1_\psi(h)=S\Theta^1_{\psi^{G^{\ast}}}(h^{G^\ast})=S\Theta^2_{\psi^{G^\ast}}(h^{G^\ast})=S\Theta^2_\psi(h)
\end{equation}
for all $h \in \ms{H}(G(F))$ and where the superscripts correspond to those of $\Pi^i$. By independence of characters, this implies that $\Pi^1_{(G,z_\iso)}(\psi)=\Pi^2_{(G,z_\iso)}(\psi)$. It remains to show that $\iota^1_{\mf{w}_H}=\iota^2_{\mf{w}_H}$. Since each $\iota^i_{\mf{w}_G}(\pi)$ is algebraic, it suffices to show that for all $\pi\in\Pi^1_{(G,z_\iso)}(\psi)=\Pi^2_{(G,z_\iso)}(\psi)$ one has that $\langle \pi,s\rangle_{\mf{w}_H}^1=\langle \pi,s\rangle_{\mf{w}_H}^2$ for all $s\in C_\psi$. By independence of characters, it suffices to show that $\Theta^{1,s}_\psi=\Theta^{2,s}_\psi$ for all $s\in C_\psi$. By the standard bijection $(H,s,\Leta, \psi^H) \Longleftrightarrow (\psi, s)$ (cf. \cite[Prop. 2.10]{BM2}) and the \textbf{(Ext)} assumption, each such $s$ comes from an extended elliptic endoscopic datum $(H,s,\Leta)$. Hence by \textbf{(ECI')} we have reduced to the quasi-split setting. We now work in the situation when $(G,z_\iso)=(G^\ast,1)$.

Let us begin with the following lemma:

\begin{lem}\label{singlelem} Suppose that $H$ is an elliptic hyperendoscopic group of $G$ and suppose that $\Pi^1_H(\psi)$ is a singleton set $\{\pi\}$. Then, in fact, $\{\pi\}=\Pi^2_H(\psi)$.
\end{lem}
\begin{proof} Since $\{\pi\}$ is a superscuspidal packet for $\Pi^1_H$, we have by assumption \textbf{(St)} that $\Theta_\pi$ is stable. By the assumption of the theorem, $\pi$ is $\Pi^2_H$-accessible and since $\Pi^2_H$ satisfies \textbf{(AS)} (by the contents of \S4), we have $\{\pi\}=\Pi^2_H(\psi')$ for some supercuspidal $L$-parameter $\psi'$ of $H$. Then, by the assumption of the theorem we have that 
\begin{equation}
    \tr(\tau\mid (r_{-\mu}\circ \psi)(\chi_\mu))\tr(h\mid \pi)=\tr(f_{\tau,h}^\mu\mid \pi)=\tr(\tau\mid (r_{-\mu}\circ\psi')(\chi_\mu))\tr(h\mid \pi)
\end{equation}
In particular, choosing $h\in\ms{H}(K^H)$ such that $\tr(h\mid \pi)\ne 0$ and letting $\tau$ vary we deduce that 
\begin{equation}
    \tr(\tau\mid(r_{-\mu}\circ\psi)(\chi_\mu))=\tr(\tau\mid(r_{-\mu}\circ\psi')(\chi_\mu))
\end{equation}
for all $\tau\in W_E$. This implies, since $\psi$ is supercuspidal so that $r_{-\mu}\circ\psi$ and $r_{-\mu}\circ\psi'$ are semi-simple, that $r_{-\mu}\circ\psi\sim r_{-\mu}\circ\psi'$ for all $\mu\in S^H$. By our assumption that $S^H$ is sufficient, we deduce that $\psi\sim\psi'$. In particular, $\{\pi\}=\Pi^2_\psi(H)$ as desired.
\end{proof}
\begin{lem}{\label{nontrivslem}}
Let $H$ be an elliptic hyperendoscopic group for $G$. Let $\psi$ be a supercuspidal parameter for $H$ and suppose $\ov{C_{\psi}} \neq \{1\}$. If $\rho$ is an irreducible representation of $\ov{C_{\psi^H}}$ then there exists a nontrivial $\ov{s} \in \ov{C_{\psi}}$ such that the trace character $\chi_{\rho}$ of $\rho$ satisfies  $\tr(\ov{s} \mid \rho) \neq 0$.
\end{lem}
\begin{proof}
Suppose $\rho$ vanishes on all nontrivial $\ov{s}$. Then we have
\begin{equation}
    1=\langle\chi_\rho,\chi_{\rho}\rangle= \frac{1}{|\ov{C_{\psi}}|} \sum\limits_{\ov{s} \in \ov{C_{\psi}}} \chi_{\rho}(\ov{s})^2=\frac{1}{|\ov{C_{\psi}}|}\chi_{\rho}(1)^2=\frac{1}{|\ov{C_{\psi}}|}\dim(\rho)^2,
\end{equation}
so that $|\ov{C_{\psi}}|= \dim(\rho)^2$. But every irreducible representation $\rho'$ of $\ov{C_{\psi}}$ is isomorphic to an irreducible factor appearing with multiplicity $\dim(\rho')$ in the regular representation of $\ov{C_{\psi}}$, which has dimension $|\ov{C_{\psi}}|$. Hence $\rho$ must be the unique irreducible representation of $\ov{C_{\psi}}$, which implies that $\rho$ is isomorphic to the trivial representation, and hence that  $|\ov{C_{\psi}}|=1$ contrary to assumption.
\end{proof}

We now explain the proof of Theorem \ref{main} in general:

\begin{proof}(of Theorem \ref{main}) We prove this by inducting on the number of roots $k$ for elliptic hyperendoscopic groups $H$ of $G$. If $k=0$ then $H$ is a torus. Since every distribution on $H$ is stable, one deduces from assumption \textbf{(Dis)} and assumption \textbf{(St)} that $\Pi^1_H(\psi)$ is a singleton and thus we are done by Lemma \ref{singlelem}. Suppose now that the result is true for elliptic hyperendoscopic groups of $G$ with at most $k$ roots. Let $H$ be an elliptic hyperendoscopic group of $G$ with $k+1$ roots and let $\psi$ be a supercuspidal parameter of $H$. We wish to show that $\Pi^1_H(\psi)=\Pi^2_H(\psi)$. If $\Pi^1_H(\psi)$ is a singleton, then we are done again by Lemma \ref{singlelem}. Otherwise, we show that $\Pi^1_H(\psi) \subset \Pi^2_H(\psi)$, which by \textbf{(Bij)} will imply that $\Pi^1_H(\psi)=\Pi^2_H(\psi)$. By Lemma \ref{nontrivslem}, we can find a non-trivial $\ov{s}\in \ov{C_\psi}$ and a lift $s \in C_{\psi}$ such that $\langle \pi,s\rangle\ne 0$. By definition of $\ov{C_{\psi}}$, we have that $s \notin Z(\widehat{G})$. Now, it suffices to show that $\Theta^{1,s}_\psi=\Theta^{2,s}_\psi$ since then by indpendence of characters, we deduce that $\pi\in\Pi^2_H(\psi)$ as desired. 

To show that $\Theta^{1,s}_\psi=\Theta^{2,s}_\psi$ for all non-trivial $s\in \ov{C_\psi}$ we proceed as follows. We obtain, by combining our assumption \textbf{(Ext)} and \cite[Proposition I.2.15]{BMY1}) from $(\psi,s)$, an extended elliptic endoscopic quadruple $(H',s,\Leta,\psi^{H'})$ with $\psi^{H'}$ supercuspidal so that $\psi=\Leta\circ\psi^{H'}$. One then has from Assumption \textbf{(ECI)} that 
\begin{equation}
    \Theta^{1,s}_\psi=\Theta^{2,s}_\psi\Longleftrightarrow S\Theta^1_{\psi^{H'}}=S\Theta^2_{\psi^{H'}}
\end{equation}
Moreover, since $s$ is non-central, we know that $H'$ has a smaller number of roots than $H$ and thus $S\Theta^1_{\psi^{H'}}=S\Theta^2_{\psi^{H'}}$ by induction. The conclusion that $\Pi^1=\Pi^2$ follows.

Let us now show that for any supercuspidal $L$-parameter $\psi$ one has that $\iota_{\mf{w}_H}^1=\iota_{\mf{w}_H}^2$ for all elliptic hyperendoscopic groups $H$ of $G$ and Whittaker data $\mf{w}_H$ of $H$. It suffices to show that $\langle \pi,s\rangle_{\mf{w}_H}^1=\langle \pi,s\rangle_{\mf{w}_H}^2$ for all $\pi\in\Pi^1_\psi(H)=\Pi^2_\psi(H)$. By independence of characters, it suffices to show that $\Theta^{1,s}_\psi=\Theta^{2,s}_\psi$ for all $s\in \ov{C_\psi}$. Since $s\in\ov{C_\psi}$, there exists, associated to the pair $(\psi,s)$, a quadruple $(H',s,\Leta,\psi^{H'})$ as in \cite[Proposition I.2.15]{BMY1} (again using also assumption \textbf{(Ext)}) where $H'$ is an elliptic endoscopic group of $H$ and $\psi^{H'}$ is a parameter such that $\psi=\Leta\circ\psi^{H'}$. By assumption \textbf{(ECI)} it suffices to show that $S\Theta^1_{\psi^{H'}}=S\Theta^2_{\psi^{H'}}$, but this follows from the previous part of the argument since we know that $\Pi^1_{H'}({\psi^{H'}})=\Pi^2_{H'}(\psi^{H'})$. The theorem follows.
\end{proof}

\section{Examples and Discussion of Assumptions}

In this section we discuss examples of groups and correspondences satisfying the assumptions \textbf{(Mu)}, \textbf{(Ext)}, and \textbf{(SS)}. We also comment on some possible rephrasings and generalizations of this work.

\subsection{Examples satisfying \textbf{(Mu)}}\label{goodsec}

In this subsection we explain that several classes of classical groups satisfy assumption \textbf{(Mu)}. In particular, we have the following:

\begin{prop}[{\cite[Theorem 8.1]{GGP}}]\label{ggpthm} Let $G$ be such that $G^\ast$ is one of the following: a general linear group, a unitary group, or an odd special orthogonal group. Then $G$ satisfies assumption \textbf{\emph{(Mu)}}.
\end{prop}

We record the following trivial observation:

\begin{lem}\label{prod} Suppose that $G_1,\ldots,G_m$ are groups satisfying assumption \textbf{\emph{(Mu)}}, then $G_1\times \cdots\times G_m$ satisfies assumption \textbf{\emph{(Mu)}}.
\end{lem}

We can then quickly explain the proof of Proposition \ref{ggpthm}:

\begin{proof} By \cite[Theorem 8.1]{GGP}, we can recover $\psi$ from $r_{- \mu} \circ \psi$ in the case of $\GL_n$, $U(n)$, $\mathrm{SO}_{2n+1}$ where $r_{-\mu}$ corresponds to the standard representation. However, to prove \textbf{(Mu)}, one must also prove a result about recovering $\psi$ from $r_{-\mu}\circ\psi$ not only for $G^\ast$ but for all elliptic hyperendoscopic groups $H$ of $G^\ast$. We analyze this for each case.
\begin{itemize}
\item There are no nontrivial elliptic hyperendoscopic groups for $\GL_n$  and so there is no difficulty in this case. 
\item The elliptic hyperendoscopic groups of unitary groups are products of unitary groups so we are done by Lemma \ref{prod}.
\item The elliptic hyperendoscopic groups of odd special orthogonal groups are products of odd special orthogonal groups so we are again done by \ref{prod}.
\end{itemize}
\end{proof}

\subsection{Examples satisfying \textbf{(Ext)}}\label{extcondsec}

The authors are not aware of any example for $G$ a group over $F$ where this property does not hold. If $G$ and all its hyperendoscopic groups have simply connected derived subgroup, then \textbf{(Ext)} follows from \cite[Prop. 1]{LanglandsConj}. In particular, unitary groups satisfy \textbf{(Ext)}. 

All elliptic endoscopic data $(H,s,\eta)$ for $G$ a symplectic or special orthogonal group can also be extended to a datum $(H, s, \Leta)$ (\cite[pg.5]{KalethaLLNQSG}). Since the elliptic endoscopic groups of symplectic and special orthogonal groups are products of groups of this type (\cite[\S 1.8]{Waldspurgerendoscopy}), it follows that symplectic and special orthogonal groups also satisfy \textbf{(Ext)}.

One could likely remove the assumption \textbf{(Ext)} altogether at the cost of having to consider $z$-extensions of endoscopic groups (see \cite{KS}) and perhaps slightly modify the statement of Theorem \ref{main} to account for these extra groups.

\subsection{Examples of Scholze--Shin datum}

In this subsection we explain the origin of the Scholze--Shin datum and equations, explain the extent to which such datum and equations are thought to exist, and discuss known examples and expected examples. 

In \cite{ScholzeDeformation}, Scholze constructs functions $f_{\tau,h}^\mu$ for certain unramified groups $G$ and for certain cocharacters $\mu$ which together form a datum $(G,\mu)$ that one might call of `PEL type'. In the second named author's thesis \cite{Youcis} these functions and their basic properties were extended to a larger class of pairs $(G,\mu)$ which the author calls `abelian type'. It seems plausible that functions $f_{\tau,h}^\mu$ (and thus Scholze--Shin datum) can be constructed in essentially full generality using the ideas of \cite{ScholzeDeformation} and \cite{Youcis} but using the moduli spaces of shtukas constructed by Scholze et al.

In \cite{ScholzeShin}, Scholze and Shin, in the course of studying the cohomology of compact unitary similitude Shimura varieties, posit that for the class of groups $G$ showing up in the pair \cite{ScholzeDeformation} that the local Langlands conjecture should satisfy the Scholze--Shin equations for the Scholze--Shin datum constructed in ibid.\textemdash we refer to such conjectures as the \emph{Scholze--Shin conjectures}. In fact, Scholze and Shin describe endoscopic versions of the Scholze--Shin equations and thus arrive at endoscopic versions of the Scholze--Shin conjectures. They show that their conjectures hold in the case of groups of the form $\GL_n(F)$ (cf. \cite{ScholzeMasters}) and the Harris--Taylor version of the local Langlands conjectures.

Using the functions $f_{\tau,h}^\mu$ constructed by the second named author in his thesis \cite{Youcis} one can construct Scholze--Shin data for a wider class of groups including the groups $G=U(n)^\ast_{E/F}$ where $E/\Q_p$ is an unramified extension of $\Q_p$. In the article \cite{BMY1}, the authors show that the Scholze--Shin conjectures hold true for such unitary groups (at least in the trivial endoscopic case which is all that is needed in this work) using the version of the local Langlands conjectures constructed by Mok in \cite{Mok} and for non-quasisplit unitary groups in \cite{KMSW}.

\subsection{Discussion of Extended Pure Inner Twists}

In this paper we have considered only $G$ that arise as extended pure inner twists of $G^*$ (e.g. see \cite{KalethaLLNQSG}). In general, the map 
\begin{equation}
    \mb{B}(G^*)_{\bas} \to \Inn(G^*),
\end{equation}
where $\Inn(G^*) := \im[H^1(F, G^*_{\ad}(\ov{F})) \to H^1(F, \Aut(G^*)(\ov{F})]$ denotes the set of inner twists of $G^*$, need not be surjective. However, when $G^*$ has connected center, this map will be surjective (see \cite[pg.20]{KalethaLLNQSG}). In general, one can likely consider all inner twists by adapting the arguments of this paper to the language of rigid inner twists as in \cite{KalethaRigid} (cf. \cite{KalethaLLNQSG}).

\subsection{The characterization in the unitary case}

Combining the discussion of \S5.1-\S5.3 and admitting the results of the preprint \cite{BMY1} we see in particular the following:

\begin{thm} Let $E/\Q_p$ be an unramified extension and $F$ the quadratic subextension of $E$. Let $G$ be an extended pure inner form twist of the quasi-split unitary group $U_{E/F}(n)^\ast$  associated to $E/F$. Then, the local Langlands correspondence for $G$ as in \cite{Mok} and \cite{KMSW} satisfies the Scholze--Shin conjecture (by the results of \cite{BMY1}) and thus is characterized by Theorem \ref{main}.
\end{thm}

\bibliographystyle{amsalpha}
\bibliography{bib}

\end{document}